\documentclass[conference, 10pt, twocolumn]{ieeeconf}       
\bstctlcite{bstctl:etal, bstctl:nodash, bstctl:simpurl}
\def\BibTeX{{\rm B\kern-.05em{\sc i\kern-.025em b}\kern-.08em
    T\kern-.1667em\lower.7ex\hbox{E}\kern-.125emX}}
    
\usepackage[left=54pt, right=54pt,bottom=54pt, top=54pt]{geometry}

\usepackage{amsmath,mathrsfs,amsfonts,amssymb,graphicx,epsfig}
 \usepackage{amsthm}
\usepackage{subcaption}
\usepackage{color,multirow,rotating}
\usepackage{algorithm,algpseudocode,algorithmicx}
\usepackage{cite,url,framed,bm,balance,nicematrix}

\newtheorem{theorem}{Theorem}

\newtheorem{remark}{Remark}

\usepackage{tikz}
\usetikzlibrary{calc,trees,positioning,arrows,chains,shapes.geometric,decorations.pathreplacing,decorations.pathmorphing,shapes, matrix,shapes.symbols}

\newcommand{\differential}{{\rm{d}}}
\renewcommand{\det}{{\mathrm{det}}}

\makeatletter
\newcommand{\dotminus}{\mathbin{\text{\@dotminus}}}

\newcommand{\@dotminus}{%
  \ooalign{\hidewidth\raise1ex\hbox{.}\hidewidth\cr$\m@th-$\cr}%
}

\allowdisplaybreaks

\title{\LARGE\textbf{
On the Contraction Coefficient of the Schr\"{o}dinger Bridge for Stochastic Linear Systems}
}

\author{Alexis M.H. Teter, Yongxin Chen, Abhishek Halder
\thanks{Alexis M.H. Teter is with the Department of Applied Mathematics, University of California, Santa Cruz, CA 95064, USA, {\tt\small{amteter@ucsc.edu}}.}%
\thanks{Yongxin Chen is with the Department of Aerospace Engineering, Georgia Institute of Technology, Atlanta, GA 30332, USA, {\tt\small{yongchen@gatech.edu}}.}%
\thanks{Abhishek Halder is with the Department of Aerospace Engineering, Iowa State University, Ames, IA 50011, USA, {\tt\small{ahalder@iastate.edu}}.}%
\thanks{This research was supported by NSF award 2112755.}
}

\IEEEoverridecommandlockouts
\begin{document}
\bstctlcite{IEEE_b:BSTcontrol}

\maketitle
\thispagestyle{empty}
\pagestyle{empty}

\begin{abstract}
Schr\"{o}dinger bridge is a stochastic optimal control problem to steer a given initial state density to another, subject to controlled diffusion and deadline constraints. A popular method to numerically solve the Schr\"{o}dinger bridge problems, in both classical and in the linear system settings, is via contractive fixed point recursions. These recursions can be seen as dynamic versions of the well-known Sinkhorn iterations, and under mild assumptions, they solve the so-called Schr\"{o}dinger systems with guaranteed linear convergence. In this work, we study \emph{a priori} estimates for the contraction coefficients associated with the convergence of respective Schr\"{o}dinger systems. We provide new geometric and control-theoretic interpretations for the same. Building on these newfound interpretations, we point out the possibility of improved computation for the worst-case contraction coefficients of linear SBPs by preconditioning the endpoint support sets.
\end{abstract}


\section{Introduction}\label{sec:introduction}
The Schr\"{o}dinger bridge problem (SBP) for a stochastic linear system concerns with minimum effort steering of 
a controlled stochastic process $\bm{x}^{\bm{u}}(t)$ satisfying It\^{o} diffusion
\begin{align}
\differential\bm{x}^{\bm{u}}(t) = \left(\bm{A}(t)\bm{x}^{\bm{u}} + \bm{B}(t)\bm{u}\right)\differential t + \sqrt{2\varepsilon}\bm{B}(t)\differential\bm{w}(t),
\label{ControlledLinearSystem}
\end{align}
over a fixed time horizon $[0,1]$, from a given initial state PDF $\rho_{0}(\cdot):=\rho^{\bm{u}}(t=0,\cdot)$ to another given terminal state PDF $\rho_{1}(\cdot):=\rho^{\bm{u}}(t=1,\cdot)$. 

In \eqref{ControlledLinearSystem}, $\bm{x}^{u}\in\mathbb{R}^{n}$ is the controlled state, $\bm{u}\in\mathbb{R}^{m}$ is the (to-be-designed) control, and $\bm{w}\in\mathbb{R}^{m}$ is the standard Wiener process. The constant $\varepsilon>0$ denotes the strength of the process noise, and is not necessarily small. 

Notice that the process noise in \eqref{ControlledLinearSystem} enters through the same ``channels" as the input, which is the case in many practical applications, e.g., in noisy actuators, in external disturbances such as wind gust affecting the system state via forcing, and in unmodeled dynamics.

We assume that\\
\textbf{A1.} the matricial trajectory pair $(\bm{A}(t),\bm{B}(t))$ is continuous and bounded for all $t\in[0,1]$,\\
\noindent \textbf{A2.} $(\bm{A}(t),\bm{B}(t))$ is a controllable pair in the sense that the finite horizon \emph{controllability Gramian}
\begin{align}
\bm{M}_{10} := \displaystyle\int_{0}^{1}\bm{\Phi}_{1\tau}\bm{B}(\tau)\bm{B}^{\top}(\tau)\bm{\Phi}^{\top}_{1\tau}\differential\tau
\label{ControllabGramian}
\end{align}
is symmetric positive definite, where $\bm{\Phi}_{t\tau}:=\bm{\Phi}(t,\tau)$ for $0\leq \tau \leq t \leq 1$ denotes the state transition matrix associated with the state matrix $\bm{A}(t)$,\\
\noindent \textbf{A3.} the given endpoint PDFs $\rho_0,\rho_1$ have compact supports $\mathcal{X}_0,\mathcal{X}_1\subset\mathbb{R}^{n}$, respectively, satisfying $\int_{\mathcal{X}_0}\rho_0 = \int_{\mathcal{X}_1}\rho_1 = 1$.

The minimum effort objective translates to the following stochastic optimal control problem:
\begin{subequations}
\begin{align}
&\underset{\bm{u}}{\arg\inf}\quad\mathbb{E}_{\bm{x}^{\bm{u}}}\int_{0}^{1}\|\bm{u}\|_2^2\:\differential t\label{LinSBPObj}\\
&\text{subject to}\quad\eqref{ControlledLinearSystem}, \; \bm{x}^{\bm{u}}(t=0)\sim\rho_0, \; \bm{x}^{\bm{u}}(t=1)\sim\rho_1,\label{LinSBPConstr}
\end{align}
\label{LinearSBP}
\end{subequations}
where the expectation in \eqref{LinSBPObj} is w.r.t. the controlled stochastic state $\bm{x}^{\bm{u}}\sim\rho^{
\bm{u}}(t,\cdot)$. The problem \eqref{LinearSBP} is to be solved over the feasible set $\mathcal{U}$ that comprises of finite energy Markovian control policies, i.e., 
$$\mathcal{U}:=\{\bm{u}:[0,1]\times\mathbb{R}^{n}\mapsto\mathbb{R}^{m}\mid \mathbb{E}\int_{0}^{1}\|\bm{u}\|_2^2\differential t < \infty\}.$$
With slight abuse of nomenclature, we will refer to problem \eqref{LinearSBP} as the ``linear SBP".

The classical SBP \cite{schrodinger1931umkehrung,schrodinger1932theorie,wakolbinger1990schrodinger,leonard2012schrodinger} is the following special case of problem \eqref{LinearSBP}: $\bm{A}(t)\equiv\bm{0}, \bm{B}(t)\equiv \bm{I}_{n}$. For recent works elaborating connections between the SBP and stochastic optimal control, see \cite{chen2016relation,chen2021stochastic}.

\begin{figure}[t]
\centering
\includegraphics[width=0.95\linewidth]{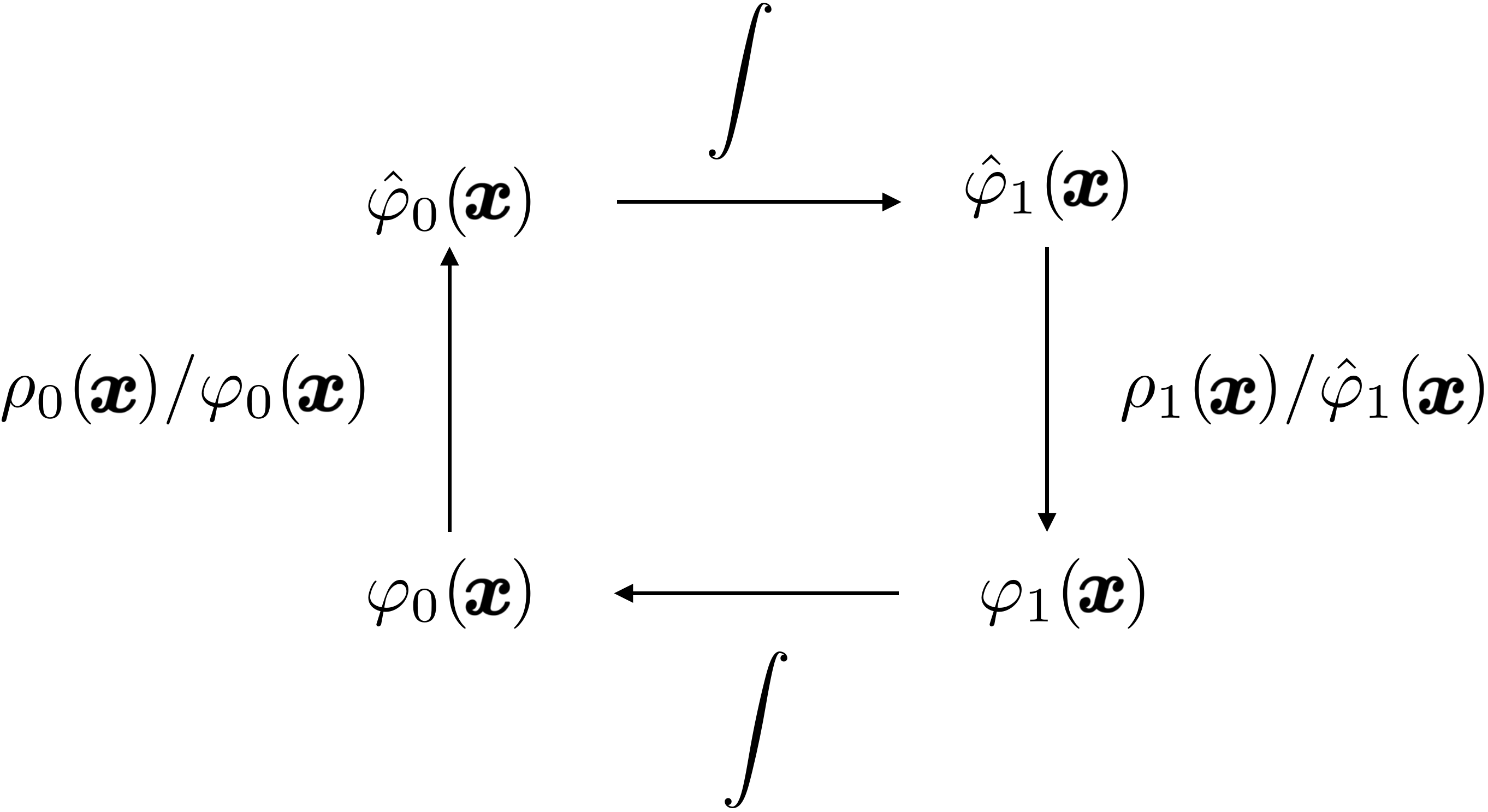}
\caption{An SBP over time horizon $[0,1]$ is solved via contractive fixed point recursion shown over the pair $\left(\hat{\varphi}_{0}(\cdot),\varphi_{1}(\cdot)\right)$. The recursion solves the associated Schr\"{o}dinger system \eqref{SchrodingerSystem}. The top (resp. bottom) horizontal arrow computes the integral in \eqref{forward} (resp. \eqref{backward}). The vertical arrows apply given boundary conditions $\rho_0,\rho_1$.}
\vspace*{-0.1in}
\label{CircularDiagm}
\end{figure}

\subsubsection{Schr\"{o}dinger system, Schr\"{o}dinger factors, and the solution of SBP}
As is well-known \cite[Sec. 8]{chen2021stochastic}, \cite[Sec. II]{caluya2021wasserstein}, SBPs such as \eqref{LinearSBP} can be solved by computing the function pair $\left(\hat{\varphi}_{0}(\cdot),\varphi_{1}(\cdot)\right)$ satisfying a system of nonlinear integral
equations, referred to as the \emph{Schr\"{o}dinger system}:
\begin{subequations}
\begin{align}
& \rho_0(\boldsymbol{x})=\hat{\varphi}_0(\boldsymbol{x}) \int_{\mathbb{R}^n} q(0, \boldsymbol{x}, 1, \boldsymbol{y}) \varphi_1(\boldsymbol{y}) \mathrm{d} \boldsymbol{y}, \label{backward}\\
& \rho_1(\boldsymbol{x})=\varphi_1(\boldsymbol{x}) \int_{\mathbb{R}^n} q(0, \boldsymbol{y}, 1, \boldsymbol{x}) \hat{\varphi}_0(\boldsymbol{y}) \mathrm{d} \boldsymbol{y},\label{forward}
\end{align}
\label{SchrodingerSystem}
\end{subequations}
where $q$ is the \emph{uncontrolled} Markov kernel associated with \eqref{ControlledLinearSystem}, i.e., the Markov kernel with $\bm{u}\equiv\bm{0}$. The system \eqref{SchrodingerSystem} can in turn be solved via a fixed point recursion over $\left(\hat{\varphi}_{0}(\cdot),\varphi_{1}(\cdot)\right)$ shown in Fig. \ref{CircularDiagm}, which is provably contractive \cite[Sec. III]{chen2016entropic} in Hilbert's projective metric \cite{hilbert1895gerade,bushell1973hilbert}. See also \cite{franklin1989scaling}.

$\left(\hat{\varphi}_{0}(\cdot),\varphi_{1}(\cdot)\right)$, thus computed, are used to find the \emph{Schr\"{o}dinger factors}
\begin{subequations}
\begin{align}
& \hat{\varphi}(t,\boldsymbol{x}):=\int_{\mathbb{R}^n} q(0, \boldsymbol{y}, t, \boldsymbol{x}) \hat{\varphi}_0(\boldsymbol{y}) \mathrm{d} \boldsymbol{y}, \quad t \geq 0,\\
& \varphi(t,\boldsymbol{x}):=\int_{\mathbb{R}^n} q(t, \boldsymbol{x}, 1, \boldsymbol{y}) \varphi_1(\boldsymbol{y}) \mathrm{d} \boldsymbol{y}, \quad t \leq 1,
\end{align}
\label{SchrodingerFactors}
\end{subequations}
which then yields the optimally controlled state PDF $\rho^{\bm{u}}_{\rm{opt}}(t,\bm{x}) = \hat{\varphi}(t,\boldsymbol{x})\varphi(t,\boldsymbol{x})$, and the optimal control $\bm{u}_{\rm{opt}}(t,\bm{x})=2\varepsilon\bm{B}(t)^{\top}\nabla_{\bm{x}}\log\varphi(t,\boldsymbol{x})$ for all $t\in[0,1]$. In particular, \eqref{SchrodingerFactors} shows that the factor $\hat{\varphi}(t,\cdot)$ (resp. $\varphi(t,\cdot)$) is $q$\emph{-harmonic} solving a backward Kolmogorov PDE (resp. $q$\emph{-coharmonic} solving a forward Kolmogorov PDE).

This result, on one hand, guarantees the existence and uniqueness of solution for the SBP. On the other hand, it offers a practical algorithm in the form of a cone-preserving fixed point recursion shown in Fig. \ref{CircularDiagm}. Thanks to the Banach contraction mapping theorem, this recursion has guaranteed \emph{linear} convergence with a \emph{contraction coefficient} $\kappa\in(0,1)$. The smaller the $\kappa$, the faster the convergence.

For fixed SBP data $\mathcal{X}_0,\mathcal{X}_{1},\varepsilon,\bm{A}(t),\bm{B}(t)$, the contraction coefficient $\kappa$ in general, depends on the specific choices of $\rho_0,\rho_1$. In this work, we focus on the \emph{worst-case contraction coefficient} $\gamma$ satisfying $\kappa\leq \gamma$. The ``worst-case" is understood over all possible $\rho_0,\rho_1$ supported on given compact sets $\mathcal{X}_0,\mathcal{X}_{1}\subset\mathbb{R}^{n}$. In other words, $\gamma$ is the tightest upper bound on $\kappa$ that only depends on $\mathcal{X}_0,\mathcal{X}_{1},\varepsilon,\bm{A}(t),\bm{B}(t)$, i.e., on the geometry of the endpoint supports, and the drift and diffusion parameters in \eqref{ControlledLinearSystem}. 

\subsubsection{Support Function}\label{subsubsec:supportfunction}
Let $\langle\cdot,\cdot\rangle$ denote the standard Euclidean inner product. We will need the notion of support function $h_{\mathcal{K}}(\cdot)$ of closed convex set $\mathcal{K}$, defined as 
\begin{align}
h_{\mathcal{K}}(\bm{y}) := \underset{\bm{x}\in\mathcal{K}}{\sup}\langle\bm{y},\bm{x}\rangle, \quad \bm{y}\in\mathbb{R}^{n},
\label{DefSptFn}
\end{align}
which measures the distance from the origin to a supporting hyperplane of $\mathcal{K}$ with normal along $\bm{y}$. From \eqref{DefSptFn}, the support function is positive homogeneous of degree one, i.e., $h_{\mathcal{K}}(a\bm{y}) = ah_{\mathcal{K}}(\bm{y})$ for $a>0$. Since only the direction of $\bm{y}$ matters, it is customary \cite[p. 209]{hiriart1996convex} to consider $\bm{y}$ as a unit vector, i.e., to restrict the domain of $h_{\mathcal{K}}(\cdot)$ to $\mathbb{S}^{n-1}$.

From \eqref{DefSptFn}, the support function is finite for a bounded set. It uniquely characterizes a convex set $\mathcal{K}$ since $h_{\mathcal{K}}(\cdot)$ equals the Legendre-Fenchel conjugate of the indicator function of $\mathcal{K}$. Definition \eqref{DefSptFn} can be extended to nonconvex $\mathcal{K}$ in the sense $h_{\mathcal{K}}(\cdot)$ is invariant under closure of convexification of $\mathcal{K}$.

The function $h_{\mathcal{K}}(\cdot)$ has nice properties that will find use in Sec. \ref{secComputationConvex}. The function is distributive under Minkowski sum: $h_{\mathcal{K}_1 + \mathcal{K}_2}(\cdot) = h_{\mathcal{K}_1}(\cdot)+h_{\mathcal{K}_2}(\cdot)$. Given convex set $\mathcal{K}$ and $\bm{T}\in\mathbb{R}^{n\times n}$, $\bm{\tau}\in\mathbb{R}^{n}$, the support function of the affine transformed set $\bm{T}\mathcal{K}+\bm{\tau}$ is  
\begin{align}
h_{\bm{T}\mathcal{K}+\bm{\tau}}(\bm{y}) = h_{\mathcal{K}}\left(\bm{T}^{\top}\bm{y}\right) + \langle\bm{\tau},\bm{y}\rangle.
\label{SptFnOfAffineTransformation}
\end{align}

\subsubsection{Contributions}\label{subsubsec:contributions}
In practice, the contraction coefficient $\kappa$ is numerically observed to be small (i.e., fast convergence) even in more general settings \cite[Fig. 4]{caluya2020reflected}, \cite{de2021diffusion,caluya2021wasserstein,haddad2020prediction} than classical or linear SBPs. The recent work \cite{stromme2023sampling} investigates classical SBP convergence from a sample complexity perspective. To the best of the authors' knowledge, this is the first work analyzing and interpreting the worst-case contraction coefficient in terms of the SBP data.

Our specific contributions are threefold:
\begin{itemize}
    \item derivation of a formula for the worst-case contraction coefficient for linear SBP in terms of the problem data $\mathcal{X}_0,\mathcal{X}_{1},\varepsilon,\bm{A}(t),\bm{B}(t)$, 
    
    \item novel control-theoretic as well as geometric interpretations for the aforesaid formula,

    \item highlighting how pre-conditioning the supports $\mathcal{X}_0,\mathcal{X}_{1}$ can help reduce the worst-case contraction coefficient, thereby improving the convergence of the linear SBP.  
\end{itemize}

\subsubsection{Organization}\label{subsubsec:organization} The layout of the material is as follows. In Sec. \ref{SecClassicalSBPRate}, we express the worst-case contraction coefficient for classical SBP in a way that permits generalization. Building on this, Sec. \ref{SecLinSBPRate} presents a formula for the worst-case contraction coefficient for linear SBP (Thm. \ref{ThmgammaLinear}), and provides a control-theoretic interpretation (Sec. \ref{subsec:controltheoreticinterpretation}) for the same. Sec. \ref{secComputationConvex} focuses on the case when the compact supports $\mathcal{X}_{0},\mathcal{X}_{1}$ are convex, and gives geometric interpretations (Thm. \ref{Thm:alphatildebetatildeSptFn}) for the quantities discussed earlier. The usage of preconditioning to improve the computation of the worst-case contraction coefficient is also discussed in Sec. \ref{secComputationConvex} with an illustrative example. Sec. \ref{sec:conclusions} concludes the paper.


\section{Contraction Coefficient for Classical SBP}\label{SecClassicalSBPRate}
To prepare ground for analyzing the contraction coefficient in linear SBP, we recall corresponding ideas for the classical SBP, and express them in a way to help generalization. 

For a given $\varepsilon>0$, consider the (scaled) standard Wiener process $\sqrt{2\varepsilon}\differential\bm{w}(t)$ in $\mathbb{R}^{n}$ with the associated Markov kernel
\begin{align}
q_{\rm{B}}\left(0,\bm{x}_{0},1,\bm{x}_1\right) := \left(4\pi\varepsilon\right)^{-n/2}
\exp\left(-\dfrac{\|\bm{x}_0 - \bm{x}_1\|_2^2}{4\varepsilon}\right),
\label{BrownianKernel}
\end{align}
where the subscript ${\rm{B}}$ stands for the Brownian a.k.a. standard Wiener process. For classical SBP, the $q\equiv q_{\rm{B}}$ in \eqref{SchrodingerSystem}-\eqref{SchrodingerFactors}. 

We note that $q_{\rm{B}}$ is continuous for all $(\bm{x}_0,\bm{x}_1)\in\mathbb{R}^{n}\times\mathbb{R}^{n}$. Furthermore, for $\mathcal{X}_0,\mathcal{X}_1$ compact, there exist constants $\alpha_{\rm{B}},\beta_{\rm{B}}$ such that $0<\alpha_{\rm{B}} \leq \beta_{\rm{B}} < \infty$, and 
\begin{align}
\alpha_{\rm{B}} \leq q_{\rm{B}}\left(0,\bm{x}_{0},1,\bm{x}_1\right) \leq \beta_{\rm{B}} \quad\forall\:(\bm{x}_0,\bm{x}_1)\in \mathcal{X}_0\times\mathcal{X}_1.
\label{alphabetaBoundBrownianKernel}
\end{align}

In \cite[equation (17)]{chen2016entropic}, the rate of convergence for the Schr\"{o}dinger system associated with the classical SBP was related to the quantity
\begin{align}
\gamma_{\rm{B}} := \tanh^{2}\left(\frac{1}{2}\log\left(\dfrac{\beta_{\rm{B}}}{\alpha_{\rm{B}}}\right)\right) \in (0,1).
\label{gammaClassicalSBP}
\end{align}
Specifically, $\gamma_{\rm{B}}$ was shown \cite[Lemma 5]{chen2016entropic} be to the \emph{worst-case contraction coefficient} for a single pass of the recursion shown in Fig. \ref{CircularDiagm}.

For ensuing development, it is helpful to define
\begin{subequations}
\begin{align}
\widetilde{\alpha}_{\rm{B}} &:= \underset{\bm{x}_0\in\mathcal{X}_0, \bm{x}_1\in\mathcal{X}_1}{\max}\:\|\bm{x}_0 - \bm{x}_1\|_2^2, \label{alphatildeClassical}\\ \widetilde{\beta}_{\rm{B}} &:= \underset{\bm{x}_0\in\mathcal{X}_0, \bm{x}_1\in\mathcal{X}_1}{\min}\:\|\bm{x}_0 - \bm{x}_1\|_2^2,\label{betatildeClassical}
\end{align}
\label{defalphatildebetatilde}    
\end{subequations}
wherein the maximum and minimum are guaranteed to exist due to the compactness of $\mathcal{X}_0,\mathcal{X}_1$. From \eqref{BrownianKernel} and \eqref{alphabetaBoundBrownianKernel}, it then follows that the $\alpha_{\rm{B}},\beta_{\rm{B}}$ in \eqref{gammaClassicalSBP} can be expressed as
\begin{align}
\alpha_{\rm{B}} = \dfrac{\exp\left(-\widetilde{\alpha}_{\rm{B}}/(4\varepsilon)\right)}{\sqrt{(4\pi\varepsilon)^{n}}}, \quad \beta_{\rm{B}} = \dfrac{\exp\left(-\widetilde{\beta}_{\rm{B}}/(4\varepsilon)\right)}{\sqrt{(4\pi\varepsilon)^{n}}}.
\label{BrownianTildeAndNotilde}
\end{align}
Consequently, we can rewrite \eqref{gammaClassicalSBP} as
\begin{align}
\gamma_{\rm{B}} = \tanh^{2}\left(\dfrac{\widetilde{\alpha}_{\rm{B}}-\widetilde{\beta}_{\rm{B}}}{8\varepsilon}\right) \in (0,1).
\label{gammaBintermsoftilde}
\end{align}
\begin{remark}\label{Remark:betatildezero}
Notice from \eqref{defalphatildebetatilde} and \eqref{BrownianTildeAndNotilde} that $0\leq \widetilde{\beta}_{\rm{B}} < \widetilde{\alpha}_{\rm{B}} < \infty$ but $0<\alpha < \beta < \infty$. In particular, $\widetilde{\beta}_{\rm{B}} = 0$ if and only if $\mathcal{X}_{0}$ and $\mathcal{X}_{1}$ overlap, i.e., $\mathcal{X}_{0}\cap\mathcal{X}_{1}\neq\emptyset$. 
\end{remark}
\begin{remark}\label{RemarkBrownianContractionCoeff}
Formula \eqref{gammaBintermsoftilde} provides an explicit relation between the worst-case contraction coefficient $\gamma_{\rm{B}}$ for the classical Schr\"{o}dinger system and the problem data given by the tuple $\left(\mathcal{X}_0,\mathcal{X}_{1},\varepsilon\right)$. In the following, we investigate how the worst-case contraction coefficient $\gamma_{\rm{L}}$ associated with the Schr\"{o}dinger system for problem \eqref{LinearSBP} depend on its problem data given by the tuple $\left(\mathcal{X}_0,\mathcal{X}_{1},\varepsilon,\bm{A}(t),\bm{B}(t)\right)$.
\end{remark}


\section{Contraction Coefficient for Linear SBP}\label{SecLinSBPRate}
In this Section, we seek to generalize the development in Sec. \ref{SecClassicalSBPRate} for the linear SBP \eqref{LinearSBP}. Under the stated assumptions \textbf{A1}-\textbf{A3}, the existence-uniqueness for the solution of this problem are guaranteed, and can be computed by solving the associated Schr\"{o}dinger system \eqref{SchrodingerSystem} with $q\equiv q_{\rm{L}}$ where 
\begin{align}
&q_{\rm{L}}\left(0,\bm{x}_{0},1,\bm{x}_1\right)\nonumber\\
&:= \det\left(\bm{M}_{10}\right)^{-1/2}
q_{\rm{B}}\left(0,\bm{M}_{10}^{-1/2}\bm{\Phi}_{10}\bm{x}_0,1,\bm{M}_{10}^{-1/2}\bm{x}_{1}\right).\label{LinearKernel}
\end{align}

\subsection{Contraction Coefficient}\label{subsecContractionCoeffLinear}
Notice that the problem data for the linear SBP involves both the endpoint PDFs (with their compact supports) as well as the dynamical coefficients $\bm{A}(t),\bm{B}(t),\varepsilon$ in \eqref{ControlledLinearSystem}. Intuition suggests that the rate of convergence will differ for different choices of controllable pair $(\bm{A}(t),\bm{B}(t))$ while keeping $(\mathcal{X}_0,\mathcal{X}_1,\varepsilon)$ fixed. For instance, faster (resp. slower) convergence is expected for linear systems which are ``easier (resp. harder) to control" than others. So we anticipate that the worst-case contraction coefficient $\gamma_{\rm{L}}$ in this case will depend on the controllability Gramian \eqref{ControllabGramian}.

We have the following result.
\begin{theorem}\label{ThmgammaLinear}
Consider the linear SBP \eqref{LinearSBP} with assumptions \textbf{A1}-\textbf{A3}. The associated Schr\"{o}dinger system \eqref{SchrodingerSystem} with $q\equiv q_{\rm{L}}$ has worst-case contraction coefficient $\gamma_{\rm{L}}\in(0,1)$, given by
\begin{align}
\gamma_{\rm{L}}=\tanh^{2}\left(\dfrac{\widetilde{\alpha}_{\rm{L}}-\widetilde{\beta}_{\rm{L}}}{8\varepsilon}\right),
\label{gammaLintermsoftilde}
\end{align}
where
\begin{subequations}
\begin{align}
\!\!\widetilde{\alpha}_{\rm{L}}\!\! &:=\!\!\! \underset{\bm{x}_0\in\mathcal{X}_0, \bm{x}_1\in\mathcal{X}_1}{\max}\!\!\left(\bm{\Phi}_{10}\bm{x}_0 - \bm{x}_1\right)^{\!\top}\!\!\bm{M}_{10}^{-1}\!\left(\bm{\Phi}_{10}\bm{x}_0 - \bm{x}_1\right),\label{alphatildeLinear}\\ \!\!\widetilde{\beta}_{\rm{L}}\!\! &:=\!\!\! \underset{\bm{x}_0\in\mathcal{X}_0, \bm{x}_1\in\mathcal{X}_1}{\min}\!\!\left(\bm{\Phi}_{10}\bm{x}_0 - \bm{x}_1\right)^{\!\top}\!\!\bm{M}_{10}^{-1}\!\!\left(\bm{\Phi}_{10}\bm{x}_0 - \bm{x}_1\right).\label{betatildeLinear}
\end{align}    
\label{defLinearalphatildebetatilde}
\end{subequations}
\end{theorem}
\begin{proof}
For a given $\varepsilon>0$, consider the uncontrolled Ornstein-Uhlenbeck (OU) process with time-varying coefficients: $$\differential\bm{x}(t) = \bm{A}(t)\bm{x}(t)\differential t + \sqrt{2\varepsilon}\bm{B}(t)\differential\bm{w}(t),$$ 
with the associated Markov kernel \eqref{LinearKernel} as
\begin{align}
&q_{\rm{L}}\left(0,\bm{x}_{0},1,\bm{x}_1\right)\nonumber\\
&= \dfrac{\exp\left(-\dfrac{\left(\bm{\Phi}_{10}\bm{x}_0 - \bm{x}_1\right)^{\top}\bm{M}_{10}^{-1}\left(\bm{\Phi}_{10}\bm{x}_0 - \bm{x}_1\right)}{4\varepsilon}\right)}{\sqrt{\left(4\pi\varepsilon\right)^{n}\det(\bm{M}_{10})}}.
\label{OUKernel}
\end{align}
Recall that $\bm{M}_{10}$ being symmetric positive definite, so are its inverse $\bm{M}_{10}^{-1}$ and the principal square root $\bm{M}_{10}^{-1/2}$. Thus \eqref{OUKernel} is well-defined.

As was the case for the kernel $q_{\rm{B}}$ in \eqref{BrownianKernel}, the kernel $q_{\rm{L}}$ too is continuous and positive for all $(\bm{x}_0,\bm{x}_1)\in\mathcal{X}_0\times\mathcal{X}_{1}$. Compactness of $\mathcal{X}_0,\mathcal{X}_1$ implies that there exist constants $\alpha_{\rm{L}},\beta_{\rm{L}}$ given by \eqref{defLinearalphatildebetatilde} such that $0<\alpha_{\rm{L}} \leq \beta_{\rm{L}} < \infty$, and 
\begin{align}
\alpha_{\rm{L}} \leq q_{\rm{L}}\left(0,\bm{x}_{0},1,\bm{x}_1\right) \leq \beta_{\rm{L}} \quad\forall\:(\bm{x}_0,\bm{x}_1)\in \mathcal{X}_0\times\mathcal{X}_1.
\label{alphabetaBoundLinearKernel}
\end{align}

From \eqref{defLinearalphatildebetatilde}, \eqref{OUKernel} and \eqref{alphabetaBoundLinearKernel}, we then obtain the following analogue of
\eqref{BrownianTildeAndNotilde}:
\begin{align}
\alpha_{\rm{L}} \!=\! \dfrac{\exp\left(-\widetilde{\alpha}_{\rm{L}}/(4\varepsilon)\right)}{\sqrt{(4\pi\varepsilon)^{n}\det(\bm{M}_{10})}}, \: \beta_{\rm{L}} \!=\! \dfrac{\exp\left(-\widetilde{\beta}_{\rm{L}}/(4\varepsilon)\right)}{\sqrt{(4\pi\varepsilon)^{n}\det(\bm{M}_{10})}}.
\label{LinearTildeAndNotilde}
\end{align}
Consequently, the worst-case contraction coefficient
\begin{align}
\gamma_{\rm{L}} &= \tanh^{2}\left(\frac{1}{2}\log\left(\dfrac{\beta_{\rm{L}}}{\alpha_{\rm{L}}}\right)\right)=\tanh^{2}\left(\dfrac{\widetilde{\alpha}_{\rm{L}}-\widetilde{\beta}_{\rm{L}}}{8\varepsilon}\right).
\label{gammaLinearSBP}    
\end{align}
\end{proof}
\begin{remark}
Due to the sub-multiplicative nature of 2 norm, the objective in \eqref{defLinearalphatildebetatilde}, in general, satisfies the bound
\begin{align}
\dfrac{\|\bm{\Phi}_{10}\bm{x}_{0}-\bm{x}_{1}\|_2^2}{\lambda_{\max}\left(\bm{M}_{10}\right)} &\leq \|\bm{M}_{10}^{-1/2}\bm{\Phi}_{10}\bm{x}_0 - \bm{M}_{10}^{-1/2}\bm{x}_{1}\|_2^2\nonumber\\
&\leq \dfrac{\|\bm{\Phi}_{10}\bm{x}_{0}-\bm{x}_{1}\|_2^2}{\lambda_{\min}\left(\bm{M}_{10}\right)},
\label{WeightedNormBound}
\end{align}
where $\lambda_{\max},\lambda_{\min}$ denote the maximum and minimum eigenvalue of $\bm{M}_{10}$, respectively.
\end{remark} 

\subsection{Control-theoretic Interpretation}\label{subsec:controltheoreticinterpretation}
We note that the objective in \eqref{defLinearalphatildebetatilde} is precisely the minimum cost for the deterministic optimal control problem:
\begin{subequations}
\begin{align}
\underset{\bm{u}}{\min}\quad&\int_{0}^{1}\|\bm{u}\|_2^2\:\differential t\label{LinDOCPObj}\\
\text{subject to}\quad&\dot{\bm{x}}^{\bm{u}}=\bm{A}(t)\bm{x}^{u}+\bm{B}(t)\bm{u},\label{ControlledLTV}\\
&\bm{x}^{\bm{u}}(t=0)=\bm{x}_0, \quad \bm{x}^{\bm{u}}(t=1)=\bm{x}_1,\label{LinDOCPEndpointConstr}
\end{align}
\label{LinDOCP}
\end{subequations}
i.e., the cost for minimum effort steering of a controllable LTV system from a fixed initial state $\bm{x}_0$ to a fixed terminal state $\bm{x}_1$ over the given time horizon $[0,1]$. See e.g., \cite[p. 194]{lee1967foundations}.

For fixed $(\bm{A}(t),\bm{B}(t))$, and therefore fixed $\bm{\Phi}_{10},\bm{M}_{10}$, the optimal cost \eqref{LinDOCPObj} varies with the variation in endpoints $\bm{x}_{0}\in\mathcal{X}_0, \bm{x}_{1}\in\mathcal{X}_1$. Thus, $\widetilde{\alpha}_{\rm{L}}$ (resp. $\widetilde{\beta}_{\rm{L}}$) equals the worst-case (resp. best-case) optimal state transfer cost for the source and target supports $\mathcal{X}_0,\mathcal{X}_1$. Recall that $\tanh^{2}(\cdot)$ is increasing over positive real. Hence $\gamma_{\rm{L}}$ in \eqref{gammaLintermsoftilde} is an increasing function of the \emph{range} of optimal state transfer cost: $\widetilde{\alpha}_{\rm{L}} - \widetilde{\beta}_{\rm{L}}$.

The $\varepsilon$ in the denominator in \eqref{gammaLintermsoftilde} implies that a stronger process noise helps to reduce $\gamma_{L}$ with other parameters held fixed, thus improving the contraction coefficient, as expected.  

In the following, we provide geometric insights for \eqref{defalphatildebetatilde} and \eqref{defLinearalphatildebetatilde}. We then discuss the computation of $\gamma_{\rm{L}}$.


\section{Geometric Interpretation and Computing $\gamma_{\rm{L}}$ for Convex $\mathcal{X}_{0},\mathcal{X}_{1}$}\label{secComputationConvex}

In \eqref{defalphatildebetatilde} and \eqref{defLinearalphatildebetatilde}, the $\widetilde{\alpha}_{\rm{B}}, \widetilde{\alpha}_{\rm{L}}$ (resp. $\widetilde{\beta}_{\rm{B}}, \widetilde{\beta}_{\rm{L}}$) can be seen as the maximal (resp. minimal) separation between the sets $\mathcal{X}_{0},\mathcal{X}_{1}$ or their linear transforms.   

When the compact sets $\mathcal{X}_{0},\mathcal{X}_{1}$ are also convex, then computing the minimum values in \eqref{betatildeClassical} and \eqref{betatildeLinear}, in general, reduce to solving the ``best approximation pair" problem; see e.g.,  \cite{bauschke2004finding}. Then,  \eqref{betatildeClassical} and \eqref{betatildeLinear} can be numerically computed using the Gilbert-Johnson-Keerthi (GJK) algorithm or its improved variants \cite{gilbert1988fast,cameron1997enhancing,montanari2017improving}.

On the other hand, the maximum values in \eqref{alphatildeClassical} and \eqref{alphatildeLinear} correspond to the \emph{squared diameters} of the Cartesian products $\mathcal{X}_{0}\times\mathcal{X}_{1}$ and $\bm{M}_{10}^{-1/2}\bm{\Phi}_{10}\mathcal{X}_{0} \times \bm{M}_{10}^{-1/2}\mathcal{X}_{1}$, respectively. When $\mathcal{X}_0,\mathcal{X}_{1}$ are compact and convex, so are these Cartesian products. Therefore, the maximum values in \eqref{alphatildeClassical} and \eqref{alphatildeLinear} must be attained at the boundaries of the sets $\mathcal{X}_{0}\times\mathcal{X}_{1}$ and $\bm{M}_{10}^{-1/2}\bm{\Phi}_{10}\mathcal{X}_{0} \times \bm{M}_{10}^{-1/2}\mathcal{X}_{1}$, respectively. However, numerical computation of these maximum values can be cumbersome depending on what kind of description for the convex sets $\mathcal{X}_0,\mathcal{X}_{1}$ are available.

In Theorem \ref{Thm:alphatildebetatildeSptFn} next, we point out that for $\mathcal{X}_0,\mathcal{X}_{1}$ convex, \eqref{defalphatildebetatilde} and \eqref{defLinearalphatildebetatilde} can be expressed in terms of the support functions (see \eqref{DefSptFn}) of these sets, thus  offering geometric insights on these quantities. 

\begin{theorem}\label{Thm:alphatildebetatildeSptFn}
Consider compact convex $\mathcal{X}_0,\mathcal{X}_{1}$ with respective support functions $h_{\mathcal{X}_0}(\cdot), h_{\mathcal{X}_1}(\cdot)$. Let $\mathbb{S}^{n-1}$ denote the Euclidean unit sphere in $\mathbb{R}^{n}$. Then \eqref{defalphatildebetatilde} can be expressed as
\begin{subequations}
\begin{align}
\widetilde{\alpha}_{\rm{B}} &= \bigg\{\underset{\bm{y}\in\mathbb{S}^{n-1}}{\max}\left(h_{\mathcal{X}_0}(\bm{y}) + h_{\mathcal{X}_1}(-\bm{y})\right)\bigg\}^2, \label{alphatildeClassicalSptFn}\\ \widetilde{\beta}_{\rm{B}} &= \bigg\{\underset{\bm{y}\in\mathbb{S}^{n-1}}{\min}\left(h_{\mathcal{X}_0}(\bm{y}) + h_{\mathcal{X}_1}(-\bm{y})\right)\bigg\}^2.
\label{betatildeClassicalSptFn}
\end{align}
\label{alphatildebetatildeClassicalSptFn}   \end{subequations}
Furthermore, \eqref{defLinearalphatildebetatilde} can be expressed as
\begin{subequations}
\begin{align}
\widetilde{\alpha}_{\rm{L}}\! &= \!\bigg\{\!\underset{\bm{y}\in\mathbb{S}^{n-1}}{\max}\!\!\left(\!h_{\mathcal{X}_0}\!\!\left(\bm{\Phi}_{10}^{\top}\bm{M}_{10}^{-1/2}\bm{y}\!\right)\! + \!h_{\mathcal{X}_1}\!\!\left(\!-\bm{M}_{10}^{-1/2}\bm{y}\!\right)\!\right)\!\!\bigg\}^{\!2}, \label{alphatildeLinearSptFn}\\ \widetilde{\beta}_{\rm{L}}\! &= \!\bigg\{\!\underset{\bm{y}\in\mathbb{S}^{n-1}}{\min}\!\!\left(\!h_{\mathcal{X}_0}\!\!\left(\bm{\Phi}_{10}^{\top}\bm{M}_{10}^{-1/2}\bm{y}\!\right)\! +\! h_{\mathcal{X}_1}\!\!\left(-\bm{M}_{10}^{-1/2}\bm{y}\!\right)\!\right)\!\!\bigg\}^{\!2}.
\label{betatildeLinearSptFn}
\end{align}
\label{alphatildebetatildeLinearSptFn}   
\end{subequations}
\end{theorem}
\begin{proof}
Consider the set difference
$$\mathcal{X}_0 - \mathcal{X}_{1}:=\mathcal{X}_{0} + \left(-\mathcal{X}_{1}\right)=\{\bm{x}_0 - \bm{x}_{1} \mid \bm{x}_0\in\mathcal{X}_{0},\bm{x}_1\in\mathcal{X}_{1}\}.$$ 
Let $\mathbb{B}^{n}:={\rm{conv}}\left(\mathbb{S}^{n-1}\right)$, the convex hull of $\mathbb{S}^{n-1}$, i.e., the $n$ dimensional unit Euclidean ball.

From \eqref{alphatildeClassical}, we have
\begin{align}
\widetilde{\alpha}_{\rm{B}} &= \bigg\{\underset{\bm{x}_0\in\mathcal{X}_{0},\bm{x}_1\in\mathcal{X}_{1}}{\max}\|\bm{x}_0 - \bm{x}_{1}\|_2\bigg\}^{2} \nonumber\\
&= \bigg\{\underset{\bm{x}\in\mathcal{X}_{0}-\mathcal{X}_{1}}{\max}\|\bm{x}\|_2\bigg\}^{2}\nonumber\\
&= \bigg\{\underset{\bm{x}\in\mathcal{X}_{0}-\mathcal{X}_{1}}{\max}\bigg\langle\frac{\bm{x}}{\|\bm{x}\|_2},\bm{x}\bigg\rangle\bigg\}^{2}\nonumber\\
&= \bigg\{\underset{\bm{x}\in\mathcal{X}_{0}-\mathcal{X}_{1}}{\max}\;\underset{\bm{y}\in\mathbb{S}^{n-1}}{\max}\langle\bm{y},\bm{x}\rangle\bigg\}^{2}\nonumber\\
&= \bigg\{\underset{\bm{y}\in\mathbb{S}^{n-1}}{\max}\;\underset{\bm{x}\in\mathcal{X}_{0}-\mathcal{X}_{1}}{\max}\langle\bm{y},\bm{x}\rangle\bigg\}^{2}\nonumber\\
&= \bigg\{\underset{\bm{y}\in\mathbb{S}^{n-1}}{\max}\;h_{\mathcal{X}_{0}-\mathcal{X}_{1}}(\bm{y})\bigg\}^{2}
\label{UsingDefSptFn}\\
&= \bigg\{\underset{\bm{y}\in\mathbb{S}^{n-1}}{\max}\;\left(h_{\mathcal{X}_{0}}(\bm{y})+h_{-\mathcal{X}_{1}}(\bm{y})\right)\bigg\}^{2}\label{SptFnDistributive}.
\end{align}
where \eqref{UsingDefSptFn} follows from the definition \eqref{DefSptFn}, and \eqref{SptFnDistributive} holds because support function is distributive over Minkowski sum. Using \eqref{SptFnOfAffineTransformation}, we get $h_{-\mathcal{X}_{1}}(\bm{y}) = h_{\mathcal{X}_{1}}(-\bm{y})$, and therefore \eqref{SptFnDistributive} equals \eqref{alphatildeClassicalSptFn}.

Likewise, from \eqref{betatildeClassical}, we have
\begin{align}
\widetilde{\beta}_{\rm{B}} &= \bigg\{\underset{\bm{x}_0\in\mathcal{X}_{0},\bm{x}_1\in\mathcal{X}_{1}}{\min}\|\bm{x}_0 - \bm{x}_{1}\|_2\bigg\}^{2} \nonumber\\
&= \bigg\{-\underset{\bm{x}\in\mathcal{X}_{0}-\mathcal{X}_{1}}{\max}\left(-\|\bm{x}\|_2\right)\bigg\}^{2}\nonumber\\
&= \bigg\{\underset{\bm{x}\in\mathcal{X}_{0}-\mathcal{X}_{1}}{\max}\bigg\langle\frac{-\bm{x}}{\|\bm{x}\|_2},\bm{x}\bigg\rangle\bigg\}^{2}\nonumber\\
&= \bigg\{\underset{\bm{x}\in\mathcal{X}_{0}-\mathcal{X}_{1}}{\max}\;\underset{\bm{y}\in\mathbb{S}^{n-1}}{\min}\langle\bm{y},\bm{x}\rangle\bigg\}^{2}\label{BeforeConvexification}\\
&= \bigg\{\underset{\bm{x}\in\mathcal{X}_{0}-\mathcal{X}_{1}}{\max}\;\underset{\bm{y}\in\mathbb{B}^{n}}{\min}\langle\bm{y},\bm{x}\rangle\bigg\}^{2},\label{AfterConvexification}
\end{align}
where the last line is due to the linear objective which allows lossless convexification for the inner minimization in \eqref{BeforeConvexification}. This can be seen explicitly from the Cauchy-Schwarz inequality: $-\|\bm{y}\|_2\|\bm{x}\|_2\leq \langle \bm{y},\bm{x}\rangle$ where the equality is achieved when $\bm{y}$ is an unit vector pointing opposite to $\bm{x}$.

Since the sets $\mathcal{X}_{0}-\mathcal{X}_{1},\mathbb{B}^{n}$ are both compact convex, applying the Von Neumann minimiax theorem \cite{v1928theorie,du1995minimax}, we rewrite \eqref{AfterConvexification} as
\begin{align}
\widetilde{\beta}_{\rm{B}} = \bigg\{\!\underset{\bm{y}\in\mathbb{B}^{n}}{\min}\;\underset{\bm{x}\in\mathcal{X}_{0}-\mathcal{X}_{1}}{\max}\langle\bm{y},\bm{x}\rangle\!\bigg\}^{\!2} \!\!=\! \bigg\{\!\underset{\bm{y}\in\mathbb{B}^{n}}{\min}\;h_{\mathcal{X}_{0}-\mathcal{X}_{1}}(\bm{y})\!\bigg\}^{\!2}.
\label{ExchangeMinMax}
\end{align}
We next revert back the feasible set of the minimization in \eqref{ExchangeMinMax} to $\mathbb{S}^{n-1}$. To justify this, note that since the origin is within $\mathbb{B}^{n}$, the minimum in \eqref{ExchangeMinMax} cannot be positive. If this minimum value is zero, we can scale the $\arg\min$ to lie on the unit sphere. So it remains to consider the case when the minimum value $h_{\mathcal{X}_{0}-\mathcal{X}_{1}}(\bm{y}^{\rm{opt}})<0$, achieved by $\bm{y}^{\rm{opt}}$ with $0< \|\bm{y}^{\rm{opt}}\| = \delta < 1$. In other words, $\bm{y}^{\rm{opt}}$ is in the interior of $\mathbb{B}^{n}$. Now consider a vector $\widetilde{\bm{y}}:=\bm{y}^{\rm{opt}}/\delta \in \mathbb{S}^{n-1}$, which is feasible. Thanks to the positive homogeneity of the support function, we have
$h_{\mathcal{X}_{0}-\mathcal{X}_{1}}\left(\widetilde{\bm{y}}\right) = \frac{1}{\delta}h_{\mathcal{X}_{0}-\mathcal{X}_{1}}(\bm{y}^{\rm{opt}})$, which yields
$$h_{\mathcal{X}_{0}-\mathcal{X}_{1}}\left(\widetilde{\bm{y}}\right) = \underbrace{\frac{1}{\delta}}_{>0}\underbrace{h_{\mathcal{X}_{0}-\mathcal{X}_{1}}(\bm{y}^{\rm{opt}})}_{<0} < h_{\mathcal{X}_{0}-\mathcal{X}_{1}}(\bm{y}^{\rm{opt}}),$$
contradicting the supposition that $\bm{y}^{\rm{opt}}$ is a minimizer. So the minimizer must lie on the boundary of the feasible set $\mathbb{B}^{n}$, i.e., on $\mathbb{S}^{n-1}$. Therefore, we can express \eqref{ExchangeMinMax} as 
$$\widetilde{\beta}_{\rm{B}} \!=\! \bigg\{\!\underset{\bm{y}\in\mathbb{S}^{n-1}}{\min}\,h_{\mathcal{X}_{0}-\mathcal{X}_{1}}(\bm{y})\!\bigg\}^{\!2}\!\!=\! \bigg\{\!\underset{\bm{y}\in\mathbb{S}^{n-1}}{\min}\left(\!h_{\mathcal{X}_{0}}(\bm{y})+h_{-\mathcal{X}_{1}}(\bm{y})\!\right)\!\bigg\}^{\!2}$$
which is indeed \eqref{betatildeClassicalSptFn} since $h_{-\mathcal{X}_{1}}(\bm{y}) = h_{\mathcal{X}_{1}}(-\bm{y})$.

To derive \eqref{alphatildebetatildeLinearSptFn}, we start by rewriting \eqref{defLinearalphatildebetatilde} as
\begin{subequations}
\begin{align}
\widetilde{\alpha}_{\rm{L}} &= \underset{\bm{x}_0\in\bm{M}_{10}^{-1/2}\bm{\Phi}_{10}\mathcal{X}_0, \bm{x}_1\in\bm{M}_{10}^{-1/2}\mathcal{X}_1}{\max}\:\|\bm{x}_0 - \bm{x}_1\|_2^2 \label{alphatildeLinearNew}\\
&= \underset{\bm{x}\in\bm{M}_{10}^{-1/2}\bm{\Phi}_{10}\mathcal{X}_0-\bm{M}_{10}^{-1/2}\mathcal{X}_1}{\max}\:\|\bm{x}\|_2^{2},\nonumber\\ 
\widetilde{\beta}_{\rm{L}} &= \underset{\bm{x}_0\in\bm{M}_{10}^{-1/2}\bm{\Phi}_{10}\mathcal{X}_0, \bm{x}_1\in\bm{M}_{10}^{-1/2}\mathcal{X}_1}{\min}\:\|\bm{x}_0 - \bm{x}_1\|_2^2\label{betatildeLinearNew}\\
&=\underset{\bm{x}\in\bm{M}_{10}^{-1/2}\bm{\Phi}_{10}\mathcal{X}_0-\bm{M}_{10}^{-1/2}\mathcal{X}_1}{\min}\:\|\bm{x}\|_2^{2}.\nonumber
\end{align}    
\label{defLinearalphatildebetatildeNew}
\end{subequations}
We the follow by following the same steps as before with the additional usage of the formula \eqref{SptFnOfAffineTransformation} relating the support functions of affine transformed sets in terms of the support functions of their pre-images. This completes the proof.
\end{proof}
\begin{remark}\label{Remark:DistanceBetweenTransformedSets}
The equalities \eqref{alphatildeLinearNew} and  \eqref{betatildeLinearNew} are particularly insightful. They highlight that $\widetilde{\alpha}_{\rm{L}},\widetilde{\beta}_{\rm{L}}$ can respectively be seen as the maximal and minimal separation between the linear transformed sets $\bm{M}_{10}^{-1/2}\bm{\Phi}_{10}\mathcal{X}_0$ and $\bm{M}_{10}^{-1/2}\mathcal{X}_1$. Specializing \eqref{alphatildeLinearNew}-\eqref{betatildeLinearNew} for the classical SBP with $\bm{A}(t)\equiv \bm{0}$, $\bm{B}(t)\equiv\bm{I}$, and thus with $\bm{M}_{10}=\bm{\Phi}_{10}=\bm{I}$, recovers \eqref{defalphatildebetatilde}, which is the maximal and minimal separation between the original supports $\mathcal{X}_0,\mathcal{X}_{1}$, as expected.    
\end{remark}


\subsection{Improved Computation via Preconditioning}\label{subsecPreconditioning}

Previous works such as \cite{kuang2017preconditioning} have explored the use of preconditioning to improve the performance of optimal transport algorithms. The preconditioning procedure described in \cite{kuang2017preconditioning} transforms the measures and corresponding support sets through a deterministic map such that the preconditioned measures are moved closer together, by creating new measures with the same (zero) mean and diagonal covariance matrix. The solutions (e.g., optimal transport map, optimal coupling) to the optimal transport problem, before and after preconditioning, are related to each other in a certain way according to the preconditioning. Such a strategy can be extended to the SBP because the SBP is an entropy-regularized optimal transport problem \cite{mikami2004monge,marino2020optimal}. 

We explore the application of such a preconditioning procedure for improved computation of $\gamma_{\rm{L}}$. The following example illustrates the main idea.

\noindent\textbf{Example 1.} Consider an instance of the linear SBP \eqref{LinearSBP} with time-invariant coefficients
\begin{align*}
    \bm{A}(t) = \begin{bmatrix} 0 & 1\\ 0 & 0\end{bmatrix}, \; \; \; \bm{B}(t) = \begin{bmatrix} 0,  \\ 1 \end{bmatrix},\; \; \; \varepsilon = 0.5,
\end{align*}
i.e., \eqref{ControlledLinearSystem} is noisy double integrator 
$$\differential x_{1}^{u} = x_{2}^{u}\:\differential t, \quad \differential x_{2}^{u} = u\:\differential t + \sqrt{2\varepsilon}\:\differential w.$$
In this case,
\begin{align*}
    \bm{\Phi}_{10} = \begin{bmatrix} 1 & 1 \\ 0 & 1 \end{bmatrix}, \quad \bm{M}_{10}^{-1} = \begin{bmatrix} 12 & -6 \\ -6 & 4 \end{bmatrix}. 
\end{align*}
We consider ellipsoidal supports
$$\mathcal{X}_i=\mathcal{E}_i\left(\bm{c}_i,\bm{S}_i\right):=\{\bm{x}\in\mathbb{R}^{2}\mid (\bm{x}-\bm{c}_i)^{\top}\bm{S}_{i}^{-1}(\bm{x}-\bm{c}_i)\leq 1\}$$
$\forall i\in\{0,1\}$, with respective center vectors
$$\bm{c}_{0} := \bm{\Phi}_{10}^{-1}\bm{M}_{10}^{1/2}\begin{pmatrix}
0\\
3
\end{pmatrix}, \quad \bm{c}_1 := \bm{M}_{10}^{1/2}\begin{pmatrix}
3\\
0
\end{pmatrix},$$
and respective positive definite shape matrices
$$\bm{S}_0 := \bm{\Phi}_{10}^{-1}\bm{M}_{10}\Phi_{10}^{-\top}, \quad \bm{S}_1 := \bm{M}_{10}=\begin{bmatrix}
1/3 & 1/2\\
1/2 & 1
\end{bmatrix}.$$
Then
\begin{subequations}
\begin{align}
\bm{M}_{10}^{-1/2} \bm{\Phi}_{10} \mathcal{X}_0 = \{ (x, y)\in\mathbb{R}^{2} | x^2 + (y-3)^2 \leq 1 \}, \label{PrecondX0}\\
\bm{M}_{10}^{-1/2} \mathcal{X}_1 = \{ (x, y)\in\mathbb{R}^{2} | (x - 3)^2 + y^2 \leq 1 \}. \label{PrecondX1}
\end{align}
\label{PrecondSupportsExample}    
\end{subequations}

\begin{figure}[t]
\centering
    \includegraphics[width=0.99\linewidth]{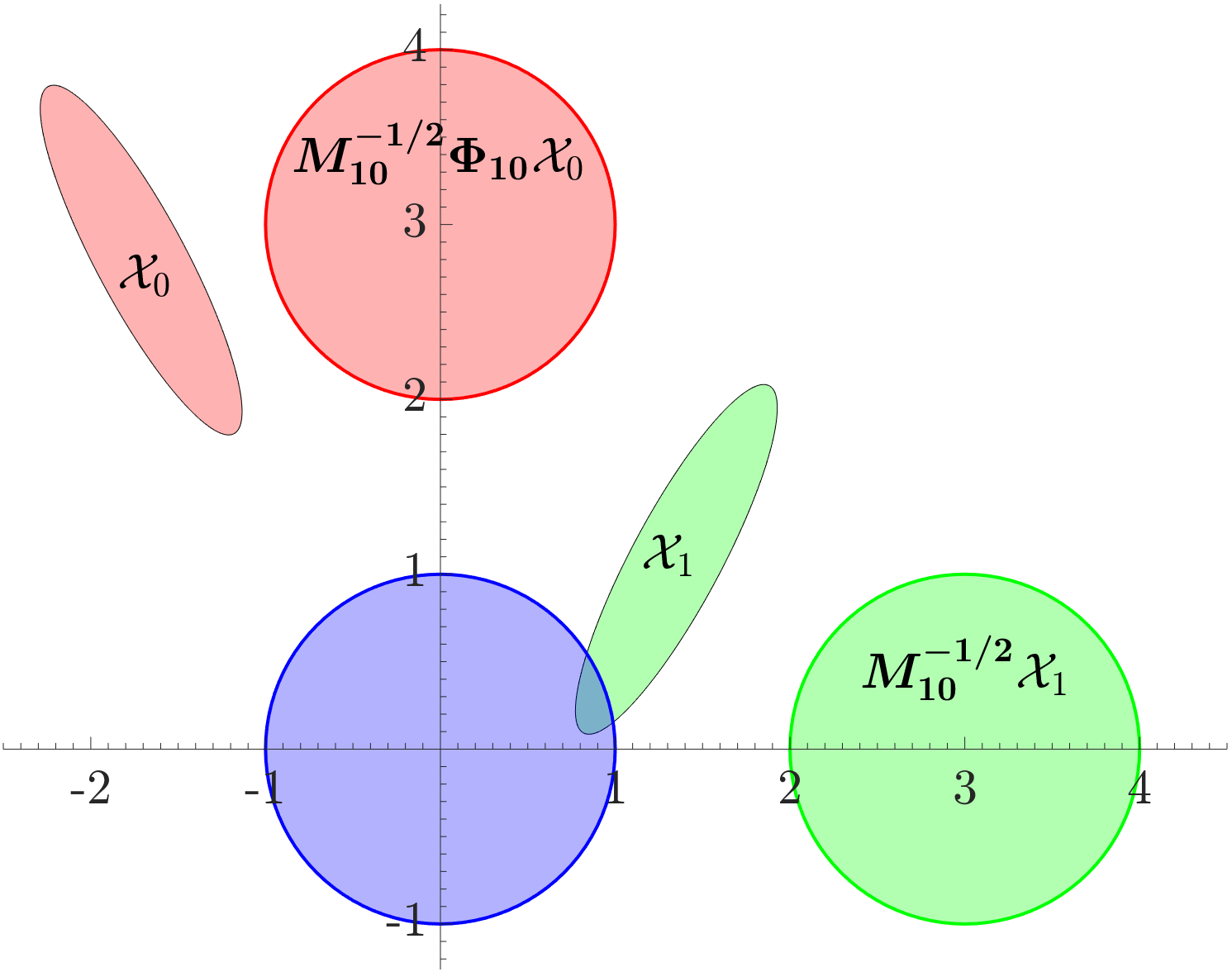}
    \caption{The sets $\mathcal{X}_0,\mathcal{X}_1$ and $\bm{M}_{10}^{-1/2}\bm{\Phi}_{10}\mathcal{X}_0,\bm{M}_{10}^{-1/2}\mathcal{X}_1$ in \textbf{Example 1}. The preconditioned supports coincide with the origin-centered unit circular disk (in \emph{blue}).}\label{fig:precond_visuals}
\end{figure}


Without the use of preconditioning, we determine $\gamma_{\rm{L}}$ from Theorem \ref{ThmgammaLinear} by considering the maximum and minimum separation between the sets $\bm{M}_{10}^{-1/2} \bm{\Phi}_{10} \mathcal{X}_0$ and $\bm{M}_{10}^{-1/2} \mathcal{X}_1$, which in our example, are two disjoint circular disks. We obtain $\widetilde{\alpha}_{\rm{L}} = 2 + 2\sqrt{3}$ and $\widetilde{\beta}_{\rm{L}} = -2 + 2\sqrt{3}$. From Theorem \ref{ThmgammaLinear}, we determine $\gamma_{{\rm{L}}} = \tanh^2(1) \approx 0.580$.

When the pushforwards $\left(\bm{M}_{10}^{-1/2} \bm{\Phi}_{10}\right)_{\sharp}\rho_0$ and $\left( \bm{M}_{10}^{-1/2}\right)_{\sharp}\rho_1$, i.e.,
$$\dfrac{\sqrt{\det(\bm{M}_{10})}}{\det\left(\bm{\Phi}_{10}\right)}\rho_0(\bm{\Phi}_{10}^{-1}\bm{M}_{10}^{1/2}(\cdot)),\, \sqrt{\det(\bm{M}_{10})}\rho_1(\bm{M}_{10}^{1/2}(\cdot)),$$ 
have identical, diagonal covariance matrices, applying the preconditioning procedure as in \cite[Sec. 5]{kuang2017preconditioning} amounts to translating the means of the supports $\bm{M}_{10}^{-1/2} \bm{\Phi}_{10} \mathcal{X}_0$ and $\bm{M}_{10}^{-1/2} \mathcal{X}_1$ to the origin. In our example, the preconditioned supports \eqref{PrecondSupportsExample} are both origin-centered unit disks (Fig. \ref{fig:precond_visuals}). Consequently, $\widetilde{\alpha}_{\rm{L}}^{{\rm{precond}}} = 2$, $\widetilde{\beta}_{\rm{L}}^{{\rm{precond}}} = 0$, and we get $\gamma_{\rm{L}}^{{\rm{precond}}} = \tanh^2(0.5) = 0.214$, which is an improvement on the original $\gamma_{\rm{L}}\approx 0.580$. 

\begin{remark}
One usage of Theorem \ref{ThmgammaLinear} is thus to demonstrate the effectiveness of proposed preconditioning procedures in reducing $\gamma_{\rm{L}}$. Additionally, the application of such preconditioning procedures can transform the supports to allow for improved calculation of $\gamma_{\rm{L}}^{{\rm{precond}}}$, as was the case in \textbf{Example 1}. How to optimally construct such a preconditioning procedure for a given SBP remains an open question that will be explored in our future work. 
\end{remark}


\section{Conclusions}\label{sec:conclusions}
This work advances systems-control-theoretic underpinnings at the intersection of Schr\"{o}dinger bridge and stochastic control problems by deriving a formula for the worst-case contraction rate for a linear SBP in terms of the problem data. The formula takes the form of squared hyperbolic tangent of a scaled range, which has nice geometric as well as optimal control-theoretic interpretations. These interpretations also suggest the possibility of preconditioning the endpoint supports for improved computation. We illustrate the same through an example, and conclude with an open question on optimal preconditioning for a given SBP.   


\bibliographystyle{IEEEtran}
\bibliography{References.bib}

\end{document}